\documentclass[11pt,a4paper]{amsart}

\usepackage[latin1]{inputenc}
\usepackage{amsmath,amsfonts,amssymb,resizegather}
\usepackage{enumerate}
\usepackage[
      colorlinks=true,    %no frame around URL
      urlcolor=blue,    %no colors
      menucolor=blue,    %no colors
      linkcolor=blue,    %no colors
      bookmarks=true,    %tree-like TOC
      bookmarksopen=true,    %expanded when starting
      hyperfootnotes=false,    %no referencing of footnotes, does not compile
      pdfpagemode=UseOutlines    %show the bookmarks when starting the pdf viewer
]{hyperref}

\newtheorem{theorem}{Theorem}[section]
\newtheorem{lemma}[theorem]{Lemma}

\newtheorem{corollary}[theorem]{Corollary}

\newtheorem{definition}{Definition}
\theoremstyle{definition}

\numberwithin{equation}{section}
\numberwithin{theorem}{section}
\numberwithin{definition}{section}

\newcommand{\R}{\mathbb{R}}
\newcommand{\Z}{\mathbb{Z}}
\newcommand{\Q}{\mathbb{Q}}
\newcommand{\C}{\mathbb{C}}

\newcommand{\cH}{\mathcal{H}}
\newcommand{\cL}{\mathcal{L}}
\newcommand{\cM}{\mathcal{M}}
\newcommand{\z}{\boldsymbol{z}}
\newcommand{\x}{\boldsymbol{x}}

\newcommand{\bo}[1]{\boldsymbol{#1}}
\newcommand{\mc}[1]{\mathcal{#1}}

\def\Vol{\textup{Vol}}
\def\lg{\left\lbrace}
\def\rg{\right\rbrace}

\newcommand{\M}{M_1}
\newcommand{\Oseen}{\mathcal{O}}
\newcommand{\ks}{k}
\newcommand{\e}{e}
\newcommand{\n}{n}
\newcommand{\en}{n}
\newcommand{\m}{m}
\newcommand{\de}{d}

\newcommand{\kbar}{\overline{\ks}}
\newcommand{\Qbar}{\overline{\Q}}
\newcommand{\td}{\text{d}}

 \author{Fabrizio Barroero}
 \email{barroero@math.tugraz.at}

\address{Institut f\"ur Mathematik A\\
Technische Universit\"at Graz\\
Steyrergasse 30, A-8010 Graz\\
Austria}

\title{Counting algebraic integers of fixed degree and bounded height}

\date{\today}

\thanks{F. Barroero is supported by the Austrian Science Foundation (FWF) project W1230-N13.}

\subjclass[2010]{Primary 11G50, 11R04.}

\keywords{Heights, algebraic integers, counting}

\begin{document}

\begin{abstract}
Let $\ks$ be a number field. For $\mathcal{H}\rightarrow \infty$, we give an asymptotic formula for 
the number of algebraic integers of absolute Weil height bounded by $\mathcal{H}$ and fixed degree over $\ks$. 
\end{abstract}

\maketitle

\section{Introduction} \label{intro}

Let $\ks$ be a number field of degree $\m$ over $\Q$.
We count the number of algebraic integers $\beta$ of degree $\e$ over $\ks$ and bounded height. 
Here and in the rest of the article, by height we mean the multiplicative height $H$ on the affine space $\Qbar^\n$ (see \cite{BombGub}, 1.5.6). 

For positive rational integers $\n$ and $\e$, and a fixed algebraic closure $\kbar$ of $\ks$, 
let
$$
\ks(\n,\e)=\{\bo{\beta}\in \kbar^\n : [\ks(\bo{\beta}):\ks]=\e \},
$$
where $\ks(\bo{\beta})$ is the field obtained by adjoining all the coordinates of $\bo{\beta}$ to $\ks$.
By Northcott's Theorem \cite{Northcott1949} any subset of $\ks(\n,\e)$ of uniformly bounded height is finite.
Therefore, for any subset $S$ of $\ks(\n,\e)$ and $\cH> 0$, we may introduce the following counting function 
$$
N(S,\cH)= |\lg \bo{\beta}\in S: H(\bo{\beta})\leq \cH  \rg|.
$$
The counting function $N(\ks(\n,\e),\cH)$ has been investigated by various people.
The best known and one of the earliest is a result of Schanuel \cite{Schanuel1979} who gave an asymptotic formula for 
$N(\ks(\n,1),\cH)$. The first who dropped the restriction of the coordinates to lie in a fix number field was Schmidt. 
In \cite{Schmidt1993}, he found upper and lower bounds for $N(\ks(\n,\e),\cH)$  and in \cite{schmidt1995} he gave an asymptotic 
formula for $N(\Q(\n,2),\cH)$. Shortly afterwards, Gao \cite{Gao1995} found the asymptotics
for  $N(\Q(\n,\e),\cH)$, provided $\n>\e$.
Later Masser and Vaaler \cite{Masser2007} established an asymptotic estimate for $N(\ks(1,\e),\cH)$. 
Finally, Widmer \cite{Widmer2009} proved an asymptotic formula for $N(\ks(\n,\e),\cH)$ for arbitrary number fields $\ks$, provided $\n>5\e/2+5+2/\m\e$. 
However, for general $\n$ and $\e$ even the correct order of magnitude for $N(\ks(\n,\e),\cH)$ remains unknown.

In this article we are interested in counting integral points, i.e., points $\bo{\beta}\in \kbar^n$, whose coordinates are algebraic integers.
Let $\Oseen_\ks$ and $\Oseen_{\kbar}$ be, respectively, the ring of algebraic integers in $\ks$ and $\kbar$. We introduce
$$
\Oseen_\ks(\n,\e)=\ks(\n,\e)\cap \Oseen_{\kbar}^\n=\{\bo{\beta}\in \Oseen_{\kbar}^\n:[\ks(\bo{\beta}):\ks]=\e\}.
$$
Possibly, the first asymptotic result (besides the trivial cases $\Oseen_\Q(\n,1)=\Z^\n$) can be found in Lang's book  \cite{Lang1983}.
Lang states, without proof,  
\begin{equation*}
N(\Oseen_\ks(1,1),\cH)=\gamma_\ks \cH^\m  \left( \log \cH  \right)^q + O\left(    {\cH}^{\m} \left( \log  \mc{H}  \right)^{q-1} \right),
\end{equation*}
where $m=[\ks:\Q]$, $q$ is the rank of the unit group of the ring of integers $\Oseen_\ks$, and $\gamma_\ks$ is an unspecified positive constant, depending on $\ks$. 
More recently,  
Widmer \cite{Widmer2013} established the following asymptotic formula 
\begin{alignat}1\label{Okne}
N(\Oseen_\ks(\n,\e),\cH)=\sum_{i=0}^{t} D_i\cH^{\m\e\n}\left(\log \cH^{\m\e \n}\right)^i+O\left(\cH^{\m\e\n-1}(\log \cH)^{t}\right),
\end{alignat}
provided $\e=1$ or $\n>\e+C_{\e,\m}$, for some explicit $C_{\e,\m}\leq 7$.
Here $t=\e(q+1)-1$, and the constants $D_i=D_i(\ks,\n,\e)$ are explicitly given. 
Widmer's result is fairly specific in the sense that he works only with the absolute multiplicative Weil height $H$.
On the other hand, the methods used in \cite{Widmer2013} are quite general and powerful, and can probably be applied
to handle other heights (such as the heights used by Masser and Vaaler in \cite{Masser2007} to deduce their main result).
As mentioned in \cite{Widmer2013} this might lead to multiterm expansions as in (\ref{Okne}) for $N(\Oseen_\ks(1,\e),\cH)$.

However, for the moment, such generalizations of (\ref{Okne}) are not available, and thus the work \cite{Widmer2013} does not provide any results in the case $\n=1$ and $\e>1$.

But Chern and Vaaler in \cite{Chern2001}, proved an asymptotic formula for the 
number of monic polynomials in $\Z[x]$ of given degree and bounded Mahler measure. Theorem 6 of \cite{Chern2001} immediately implies the following
result 
\begin{equation}\label{chern}
N(\Oseen_\Q(1,\e),\cH)= C_{\e} \cH^{\e^2}+  O\left(    \mc{H}^{\e^2-1} \right),
\end{equation}
for some explicitly given positive real constant $C_\e$. Theorem \ref{thmint}
extends Chern and Vaaler's result to arbitrary ground fields $\ks$. 

For positive rational integers $\e$ we define
$$
C_{\R,\e} = 2^{\e-M} \left( \prod_{l=1}^{M}\left( \frac{2l}{2l+1} \right)^{\e-2l} \right) \frac{e^M}{M!},
$$
with $M=\lfloor \frac{\e-1}{2} \rfloor$, and 
$$
C_{\C,\e} = \pi^\e \frac{\e^\e}{\left( \e! \right)^2}.
$$
And, finally, let
\begin{equation}\label{const}
 C^{(\e)}_{\ks}= \frac{\e^{2q+1}  2^{s\e} \m^{q}}{q!\left( \sqrt{|\Delta_\ks|}\right)^{\e}}C_{\R,\e}^r C_{\C,\e}^s,
\end{equation}
where $\m=[\ks:\Q]$, $r$ is the number of real embeddings of $\ks$, $s$ the number of pairs of complex conjugate embeddings, $q=r+s-1$, and 
$\Delta_\ks$ denotes the discriminant of $\ks$. As usual, here and in the rest of this article, the empty product is understood to be 1.

For non-negative real functions $f(X),g(X),h(X)$ and $X_0\in \R$  we write 
$f(X)=g(X)+O(h(X))$ as $X\geq X_0$ tends to infinity if there is $C_0$ such that $|f(X)-g(X)|\leq C_0 h(X)$ for all $X \geq X_0$.

\begin{theorem}\label{thmint}
Let $\e$ be a positive integer, and let $\ks$ be a number field of degree $m$ over $\Q$.
Then, as $\cH\geq 2$ tends to infinity, we have
\begin{equation*}\label{mainest}
N(\Oseen_\ks(1,\e),\cH)=C^{(\e)}_{\ks}\mc{H}^{\m\e^2} \left( \log  \mc{H} \right)^{q} +\left\lbrace
\begin{array}{ll}
O\left(    \mc{H}^{\m\e^2} \left( \log  \mc{H}  \right)^{q-1} \right),  & \mbox{ if $q\geq 1$, }\\
O\left(\mc{H}^{\e(\m\e-1)} \mc{L}  \right), & \mbox{ if $q=0$, }
\end{array} \right.
\end{equation*}
where $\mc{L}=\log \cH$ if $(m,e)=(1,2)$ and 1 otherwise.
The implicit constant in the error term depends only on $\m$ and $\e$.
\end{theorem}

Let us mention two simple examples. The number of algebraic integers $\alpha$ quadratic over $\Q(\sqrt{2})$ with $H(\alpha)\leq \cH$ is 
$$
32\cH^{8} \log \cH+O(\cH^8).
$$
In case $e=3$, we have
$$
108\sqrt{2} \cH^{18}\log \cH + O(\cH^{18})
$$
algebraic integers $\alpha$ cubic over $\Q(\sqrt{2})$ with $H(\alpha)\leq \cH$.

Our approach is similar to the one used to obtain (\ref{chern}) above, because we count monic polynomials in $\Oseen_\ks[X]$, but this is not a straightforward generalization of Theorem 6 of \cite{Chern2001}. In fact, in \cite{Chern2001} the estimate on the number of monic polynomials in $\Z[x]$ is obtained from a counting lattice points theorem, which is formulated only for the standard lattice $\Z^n$ (\cite{Chern2001}, Lemma 24).
Our proof relies on a new counting theorem for points of an arbitrary lattice in definable sets in an o-minimal structure \cite{Barroero2012}. Moreover, our proof is fairly short, and more straightforward than the approach of \cite{Widmer2013},
but to the expense that we do not get a multiterm expansion. 

In \cite{Masser2007}, Masser and Vaaler observed that the limit for $
\cH\rightarrow \infty$ of
$$
\frac{N(\ks(1,\e),\cH^{\frac{1}{\e}})}{N(\ks(\e,1),\cH)}
$$ 
is a rational number. Moreover, they asked if this can be extended to some sort of reciprocity law, i.e., whether
$$
\lim_{\cH \rightarrow \infty}\frac{N(\ks(\n,\e),\cH^{\frac{1}{\e}})}{N(\ks(\e,\n),\cH^{\frac{1}{\n}})}\in \Q.
$$
If we consider only the first term in (\ref{Okne}), and combine it with Theorem \ref{thmint} we see that 
$$
\lim_{\cH \rightarrow \infty} \frac{N(\Oseen_\ks(1,\e),\cH^{\frac{1}{\e}})}{N(\Oseen_\ks(\e,1),\cH)}=\e \left(\frac{C_{\R,\e}}{2^\e}\right)^r \left(\frac{C_{\C,\e}}{\pi^\e}\right)^s
$$
is a rational number depending only on $e$, $r$ and $s$. As Masser and Vaaler did, one can ask again whether 
$$
\lim_{\cH \rightarrow \infty}\frac{N(\Oseen_\ks(\n,\e),\cH^{\frac{1}{\e}})}{N(\Oseen_\ks(\e,\n),\cH^{\frac{1}{\n}})}\in \Q.
$$

\section{Counting monic polynomials} \label{sect2}

In this section we see how our problem translates to counting monic polynomials of fixed degree that assume a uniformly bounded value under a certain real valued function called $M^\ks$, defined using the Mahler measure.

Recall we fixed a number field $\ks$ of degree $\m$ over $\Q$ and $\Oseen_\ks$ is its ring of integers. Let $\sigma_1, \dots , \sigma_r$ be the real embeddings of $\ks$ and $\sigma_{r+1}, \dots , \sigma_{\m}$ be the strictly complex ones, indexed in such a way that $\sigma_{j}=\overline{\sigma}_{j+s}$ for $j=r+1 , \dots , r+s$. Therefore, $r$ and $s$ are, respectively, the number of real and pairs of conjugate complex embeddings of $\ks$ and $\m=r+2s$. We put $d_i=1$ for $i=1, \dots ,r$ and $d_i=2$ for $i=r+1, \dots ,r+s$ and fix a positive integer $\e$. Let us recall the definition of the Mahler measure.

\begin{definition}
If $f=z_0X^\de+z_1X^{\de-1}+\cdots +z_\de \in \C[X]$ is a non-zero polynomial of degree $\de$ with roots $\alpha_1,\ldots , \alpha_\de$, the Mahler measure of $f$ is defined to be
\begin{equation*}
M(f)=|z_0| \prod_{i=1}^\de \max\lg 1,|\alpha_i|\rg.
\end{equation*}
Moreover, we set $M(0)=0$.
\end{definition}
We see $M$ as a function $\C[X] \rightarrow [0,\infty)$ and define
\begin{equation*}
\begin{array}{cccl}\label{defmk}
M^\ks:& \ks[X]& \rightarrow & [0,\infty) \\
 & f  & \mapsto & \prod_{i=1}^{r+s}M(\sigma_i(f))^{\frac{d_i}{\m}},
\end{array}
\end{equation*}
where $\sigma_i$ acts on the coefficients of $f$. Note that, for every $\alpha \in \Oseen_\ks$, 
\begin{equation}\label{mahlheight}
M^\ks(X-\alpha)=\prod_{i=1}^{r+s}\max \lg 1,|\sigma_i(\alpha)| \rg^{\frac{d_i}{\m}}=H(\alpha).
\end{equation}
In fact, if $\alpha \in \Oseen_\ks$ then $|\alpha|_v\leq 1$ for every non-archimedean place $v$ of $\ks$.

Moreover, the Mahler measure is multiplicative by definition, i.e.,
$$
M(fg)=M(f)M(g),
$$
and one can see that
$$
M^\ks(fg)=M^\ks(f)M^\ks(g),
$$
for every $f,g \in \ks[X]$.  

For some positive integer $e$ and some $\cH>0$, we define $\cM^\ks(e,\cH)$ to be the set of monic $f \in \Oseen_\ks[X]$ of degree  $e$ and $M^\ks(f)\leq \cH$.
It is easy to see that $\cM^\ks(e,\cH)$ is finite for all $\cH$. The following theorem gives an estimate for its cardinality.

\begin{theorem}\label{mainthm}
For every $\cH_0>1$ there exists a $D_0$ such that, for every $\cH \geq \cH_0$,
\begin{equation}\label{asym}
\left| \left|\cM^\ks(e,\cH)\right|- \frac{C^{(\e)}_{\ks}}{\e^{q+1}}\mc{H}^{\m\e} \left( \log  \mc{H} \right)^{q} \right| \leq \left\lbrace
\begin{array}{ll}
D_0    \mc{H}^{\m\e} \left( \log  \mc{H}  \right)^{q-1} ,  & \mbox{ if $q\geq 1$, }\\
D_0 \mc{H}^{\m\e-1}  , & \mbox{ if $q=0$, }
\end{array} \right.
\end{equation}
where $q=r+s-1$. The constant $D_0$ depends only on $\cH_0$, $\m$ and $\e$.
\end{theorem}

Note that our constant $C^{(\e)}_{\ks}$ defined in (\ref{const}), is bounded if we fix $\m$ and $\e$ and we let $\ks$ vary among all number fields of degree $\m$. This implies that there exists a real constant $C^{(\m,\e)}$, depending only on $\m$ and $\e$, such that $ \left|\cM^\ks(e,\cH)\right|$ is bounded from above by
\begin{equation}\label{asym2}
C^{(\m,\e)}\mc{H}^{\m\e} \left( \log  \mc{H} +1 \right)^{q} ,
\end{equation}
for every $\cH \geq 1$.

We prove Theorem \ref{mainthm} later and for the rest of this section we derive Theorem \ref{thmint} from Theorem \ref{mainthm}. We follow the line of Masser and Vaaler \cite{Masser2007}.

Now we want to restrict to monic $f$ irreducible over $\ks$. Let $ \widetilde{\cM}^\ks(e,\cH)$ be the set of polynomials in $\cM^\ks(e,\cH)$ that are irreducible over $\ks$.

\begin{corollary}\label{maincor}
For every $\cH_0>1$ there exists an $F_0$ such that, for every $\cH \geq \cH_0$,
\begin{equation}\label{asym3}
\left| \left|\widetilde{\cM}^\ks(e,\cH)\right|- \frac{C^{(\e)}_{\ks}}{\e^{q+1}}\mc{H}^{\m\e} \left( \log  \mc{H} \right)^{q} \right| \leq \left\lbrace
\begin{array}{ll}
F_0    \mc{H}^{\m\e} \left( \log  \mc{H}  \right)^{q-1} ,  & \mbox{ if $q\geq 1$, }\\
F_0 \mc{H}^{\m\e-1} \cL , & \mbox{ if $q=0$, }
\end{array} \right.
\end{equation}
where $\mc{L}=\log \cH$ if $(m,e)=(1,2)$ and 1 otherwise.
The constant $F_0$ depends again only on $\cH_0$, $\m$ and $\e$.
\end{corollary}

\begin{proof}
For $e=1$ there is nothing to prove. Suppose $e>1$. We show that, up to a constant, the number of all monic reducible $f \in \Oseen_\ks[X]$ of degree $\e$ with $M^\ks(f)\leq \mc{H}$ is not larger than the right hand side of (\ref{asym}), except for the case $(m,e)=(1,2)$. 

Consider all $f=gh \in \cM^\ks(e,\cH)$ with $g,h\in \Oseen_\ks[X]$ monic of degree $a$ and $b$ respectively, with $0<a\leq b<\e$ and $a+b=e$.
We have $1 \leq M^\ks(g),M^\ks(h) \leq \cH$ because $g$ and $h$ are monic. Thus, there exists a positive integer $l$ such that $2^{l-1}\leq M^\ks(g)<2^l$. Note that $l$ must satisfy
\begin{equation}\label{ineqk}
1\leq l \leq \frac{\log \mc{H}}{\log 2}+1\leq 2 \log \mc{H} +1.
\end{equation}
Since $M^\ks$ is multiplicative,
$$
M^\ks(h)=\frac{M^\ks(f)}{M^\ks(g)}\leq 2^{1-l} \mc{H}.
$$
Using (\ref{asym2}) and noting that $2^l\leq 2 \cH$, we can say that there are at most
$$
C^{(\m,a)}\left(2^l\right)^{\m a } \left( \log  2^l+1 \right)^{q} \leq C^{(\m,a)}\left(2^l\right)^{\m a } \left( \log \cH +2 \right)^{q}  
$$
possibilities for $g$ and
$$
C^{(\m,b)}\left(2^{1-l} \mc{H}\right)^{\m b}\left( \log \left(2^{1-l} \mc{H} \right)+1 \right)^{q} \leq C^{(\m,b)}\left(2^{1-l} \mc{H}\right)^{\m b}\left( \log  \mc{H} +2 \right)^{q}
$$
possibilities for $h$. Therefore, we have at most
\begin{align} \label{red}
C' \mc{H}^{\m b} 2^{\m l(a-b)} \left(  \log \mc{H} +2 \right)^{2q}
\end{align}
possibilities for $gh$ with $M^\ks (gh)\leq \cH$ and $2^{l-1}\leq M^\ks(g)<2^l$, where $C'$ is a real constant. Since there are only finitely many choices for $a$ and $b$ we can take $C'$ to depend only on $\m$ and $\e$.

If $a=b=\frac{\e}{2}$ then (\ref{red}) is
$$
C'\mc{H}^{\m \frac{\e}{2}}  \left(  \log \mc{H} +2 \right)^{2q}.
$$
Summing over all $l$, $1\leq l \leq  \lfloor 2 \log \mc{H} \rfloor+1$ (recall (\ref{ineqk})), gives an extra factor $2\log \mc{H} +1 $. Therefore, when $a=b$, there are at most
$$
C'\mc{H}^{\frac{\m \e}{2}}\left(2 \log \mc{H} +2 \right)^{2q+1}
$$
possibilities for $f=gh$, with $M^\ks(f)\leq \cH$. If $(m,e)\neq (1,2)$, this has smaller order than the right hand side of (\ref{asym}), since $\m \e >2$ implies $\frac{\m \e}{2}< \m \e-1$. In the case $(m,e)= (1,2)$ we get $C' \cH \left(2 \log \mc{H} +2 \right)$ and we need an additional logarithm factor.

In the case $a <b$, summing $2^{\m l(a-b)}$ over all $l$, $1\leq l \leq  \lfloor 2 \log \mc{H} \rfloor+1 =L$, we get
$$
\sum_{l=1}^L \left(2^{\m(a-b)}\right)^l \leq \sum_{l=1}^{L}2^{-l}\leq 1.
$$
Thus, recalling $b\leq \e-1$, when $a<b$, there are at most 
$$
C'' \mc{H}^{\m(\e-1)} \left( \log \mc{H} +2 \right)^{2q}
$$
possibilities for $f=gh$, with $M^\ks(f)\leq \cH$, where again $C''$ depends only on $\m$ and $\e$. This is again not larger than the right hand side of (\ref{asym}).
\end{proof}

For the last step of the proof we link such monic irreducible polynomials with their roots.

\begin{lemma} \label{lemmapoly}

An algebraic integer $\beta$ has degree $\e$ over $\ks$ and $H( \beta ) \leq \cH$ if and only if it is a root of a monic irreducible polynomial $f \in \Oseen_\ks[X]$ of degree $\e$ with $M^\ks(f)\leq \cH^\e$

\end{lemma}

\begin{proof}
Suppose $f\in \Oseen_\ks[X]$ is a monic irreducible polynomial of degree $\e$ and $\beta$ is one of its roots, i.e., $\beta$ is an algebraic integer with $[\ks(\beta):\ks]=e$ and minimal polynomial $f$ over $\ks$. We claim that 
$$
M^k(f)=H(\beta)^\e.
$$

The function $M^k$ is independent of the field $\ks$ and we can define an absolute $M^{\overline{\Q}}$ over $\overline{\Q}[X]$ that, restricted to any $\ks[X]$, equals $M^\ks$. To see this one can simply imitate the proof of the fact that the Weil height is independent of the field containing the coordinates (see \cite{BombGub}, Lemma 1.5.2).
%
%
%
% Recall that $[\ks:\Q]=\m$. Let $\ks'$ be a finite extension of $\ks$ with $[\ks':\Q]=\m'$. We put $M_{\ks,\infty}$ for the set of infinite places of $\ks$. For every $w \in  M_{\ks,\infty}$, $\sigma_w$ is the corresponding embedding into $\C$ and $d_w$ is the local degree, i.e $d_w=1$ or $2$ if $w$ corresponds to a real or a strictly complex embedding, respectively. Recall (see \cite{Neukirch1999}, Ch.II, (8.4) Corollary)
%$$
%\sum_{\substack{v \in M_{\ks',\infty}\\v \mid w}} d_v=d_w[\ks':\ks]=d_w\frac{\m'}{\m}.
%$$
%For any $f \in \ks[X]$, we have
%\begin{align*}
%M^{\ks'}(f) &= \prod_{v \in M_{\ks',\infty}} M(\sigma_v (f))^\frac{d_v}{\m'}=  \prod_{w \in M_{\ks,\infty}} \prod_{\substack{v \in M_{\ks',\infty}\\v \mid w}}M(\sigma_v (f))^\frac{d_v}{\m'}=\\
%&=\prod_{w \in M_{\ks,\infty}} M(\sigma_w (f))^{\sum_{\substack{v \in M_{\ks',\infty}\\v \mid w}} \frac{d_v}{\m'}} 
%=\prod_{w \in M_{\ks,\infty}} M(\sigma_w (f))^\frac{d_w}{\m}=M^\ks(f).
%\end{align*}

Suppose $f=(X-\alpha_1)\cdots (X-\alpha_\e)$. Since the $\alpha_i$ are algebraic integers, by (\ref{mahlheight}), we have
$$
M^{\overline{\Q}}(X-\alpha_i)=M^{\Q(\alpha_i)}(X-\alpha_i)=H(\alpha_i),
$$
and the $\alpha_i$ have the same height because they are conjugate (see \cite{BombGub}, Proposition 1.5.17). Moreover, by the multiplicativity of $M^\ks$ we can see that
$$
M^\ks(f)=M^{\overline{\Q}}(f)=\prod_{i=1}^e M^{\overline{\Q}}(X-\alpha_i)= H(\alpha_j)^\e,
$$
for any $\alpha_j$ root of $f$.
\end{proof}

Lemma \ref{lemmapoly} implies that $ N(\Oseen_\ks(1,\e),\cH)= e \left|\widetilde{\cM}^\ks(e,\cH^\e) \right|$ because there are $\e$ different $\beta $ with the same minimal polynomial $f$ over $\ks$. Therefore, by (\ref{asym3}), we have that for every $\cH_0>1$ there exists a $C_0$, depending only on $\cH_0$, $\m$ and $\e$, such that for every $\cH\geq \cH_0$,
\begin{equation*}
\left| N(\Oseen_\ks(1,\e),\cH) - C^{(\e)}_{\ks} \mc{H}^{\m\e^2} \left( \log  \mc{H} \right)^{q} \right| \leq \left\lbrace
\begin{array}{ll}
C_0    \mc{H}^{\m\e^2} \left( \log  \mc{H}  \right)^{q-1} ,  & \mbox{ if $q\geq 1$, }\\
C_0 \mc{H}^{\e(\m\e-1)} \cL , & \mbox{ if $q=0$, }
\end{array} \right.
\end{equation*}
where $\mc{L}=\log \cH$ if $(m,e)=(1,2)$ and 1 otherwise. We get Theorem \ref{thmint} by choosing $\cH_0=2$.

\section{A counting principle}

In this section we introduce the counting theorem that will be used to prove Theorem \ref{mainthm}. The principle dates back to Davenport \cite{Davenport1951} and was developed by several authors. In a previous work \cite{Barroero2012} the author and Widmer formulated a counting theorem that relies on Davenport's result and uses o-minimal structures. The full generality of Theorem 1.3 of \cite{Barroero2012} is not needed here as we are going to count lattice points in semialgebraic sets.
\begin{definition}
Let $N$, $M_i$, for $i=1, \dots ,N$, be positive integers. A semialgebraic subset of $\R^n$ is a set of the form
$$
\bigcup_{i=1}^{N}\bigcap_{j=1}^{M_i} \{  \x \in \R^n : f_{i,j}(\x) \ast_{i,j} 0 \},
$$
where $f_{i,j} \in \R[X_1, \dots , X_n]$ and the $\ast_{i,j}$ are either $<$ or $=$.
\end{definition}

A very important feature of semialgebraic sets is the fact that this collection of subsets of the Euclidean spaces is closed under projections. This is the well known Tarski-Seidenberg principle.

\begin{theorem}[\cite{Bierstone88}, Theorem 1.5]\label{tarski}
Let $A \in \R^{n+1}$ be a semialgebraic set, then $\pi(A)\in \R^n$ is semialgebraic, where $\pi :\R^{n+1}\rightarrow \R^n$ is the projection map on the first $n$ coordinates. 
\end{theorem}

Let $S \subseteq \R^{n+n'}$, for a $\bo{t}\in \R^{n'}$ we call $S_{\bo{t}}=\{ \x \in \R^n: (\x,\bo{t})\in S\}$ the fiber of $S$ above $\bo{t}$. Clearly, if $S$ is semialgebraic also the fibers $S_{\bo{t}}$ are semialgebraic. If so, we call $S$ a semialgebraic family.

Let $\Lambda$ be a lattice of $\R^n$, i.e., the $\Z$-span of $n$ linearly independent vectors of $\R^n$. Let $\lambda_i=\lambda_i(\Lambda)$ for $i=1,\ldots,n$ be the successive minima of $\Lambda$ with respect to the zero centered unit ball $B_0(1)$, i.e., for $i=1,...,n$
\begin{alignat*}1
\lambda_i=\inf\{\lambda:B_0(\lambda)\cap\Lambda \text{ contains $i$ linearly independent vectors}\}.
\end{alignat*}
The following theorem is a special case of Theorem 1.3 of \cite{Barroero2012}.

\begin{theorem}\label{counttheorem}
Let $Z\subset \R^{n+n'}$ be a semialgebraic family and suppose the  fibers $Z_{\bo{t}}$ are bounded. Then there exists a constant $c_Z \in \R$, depending only on the family, such that, for every $\bo{t} \in \R^{n'}$,
\begin{equation*}\label{eqcount}
\left| |Z_{\bo{t}}\cap \Lambda|-\frac{\Vol(Z_{\bo{t}})}{\det \Lambda} \right|\leq \sum_{j=0}^{n-1}c_{Z}\frac{V_j(Z_{\bo{t}})}{\lambda_1\cdots \lambda_j},
\end{equation*}
where $V_j(Z_{\bo{t}})$ is the sum of the $j$-dimensional volumes of the orthogonal projections of $Z_{\bo{t}}$ 
on every $j$-dimensional coordinate subspace of $\R^n$ and $V_0(Z_{\bo{t}})=1$.
\end{theorem}

\section{A semialgebraic family}

In this section we introduce the family we want to apply Theorem \ref{counttheorem} to.

We see the Mahler measure as a function of the coefficients of the polynomial. We fix $n>0$ and define $M:\R^{n+1}$ or $\C^{n+1} \rightarrow [0,\infty)$ such that
$$
M(z_0,\ldots ,z_n)= M(z_0 X^n + \dots + z_n).
$$
These two functions satisfy the definition of bounded distance function in the sense of the geometry of numbers, i.e.,
\begin{enumerate}
\item $M$ is continuous;
\item $M(\bo{z})=0$ if and only if $\bo{z}=\bo{0}$;
\item $M(w\bo{z})=|w|M(\bo{z})$, for any scalar $w \in \R$ or $\C$.
\end{enumerate}

Properties (2) and (3) are obvious from the definition, while continuity was proved already by Mahler (see \cite{Mahler1961}, Lemma 1).

Let $\M$ be the monic Mahler measure function, i.e., $\M(\z)=M(1,\z)$ for $\z \in \R^\en$ or $\C^\en$.

In the following we consider the complex monic Mahler measure as a function
\begin{equation*}
\begin{array}{cccc}\label{Mahlerf}
\M: & \R^{2\en} & \rightarrow & \R \\
  & \left( x_1,\ldots ,x_{2\en}\right) & \mapsto & M\left(X^\en + (x_1+ix_2)X^{\en-1}+ \cdots +x_{2\en-1}+ix_{2\en}\right).
\end{array}
\end{equation*}

We fix positive integers $\en, \m ,r,s$ with $\m=r+2s$ and $d_1 ,\dots d_{r+s}$ such that $d_i=1$ for $i=1, \dots ,r$ and $d_i=2$ for $i=r+1, \dots ,r+s$.

We define
\begin{equation}\label{defZ}
Z=\lg (\bo{x}_1, \ldots , \bo{x}_{r+s},t) \in \left( \R^\en \right)^r \times \left( \R^{2\en} \right)^s \times \R : \prod_{i=1}^{r+s} \M(\bo{x}_i)^{d_i}\leq t  \rg.
\end{equation}
Here $\x_i \in \R^{d_i \en}$ and $\M(\bo{x}_i)$ is the real or the complex monic Mahler measure respectively if $i=1, \dots ,r$ or $i=r+1, \dots ,r+s$.

We want to count lattice points in the fibers $Z_t\subseteq \R^{mn}$ using Theorem \ref{counttheorem}, therefore we need to show that $Z$ is a semialgebraic set and that the fibers $Z_t$ are bounded.

\begin{lemma}

The set $Z$ defined in (\ref{defZ}) is semialgebraic.

\end{lemma}

\begin{proof}
Recall the definition of $Z$. To each $\x_i\in \R^{d_in}$ corresponds a monic polynomial $f_i$ of degree $n$ with real (for $i=1,\dots r$) or complex (for $i=r+1,\dots r+s$) coefficients.
Let $S$ be the set of points 
$$
\left( \bo{x}_1, \ldots , \bo{x}_{r+s},t, t_1 ,\dots ,t_{r+s}, \bo{\alpha}^{(1)},\bo{\beta}^{(1)}, \dots ,\bo{\alpha}^{(r+s)},\bo{\beta}^{(r+s)} \right)
$$
in $\R^{n(r+2s)+1+r+s+2n(r+s)}$, with $\bo{\alpha}^{(i)},  \bo{\beta}^{(i)} \in \R^n$, such that 
\begin{itemize}
\item $\bo{\alpha}^{(i)} $ and $ \bo{\beta}^{(i)}$ are, respectively, the vectors of the real and the imaginary parts of the $n$ roots of $f_i$, for every $i=1 ,\dots ,r+s$;
\item $\prod_{l=1}^{n}\max \lg 1,  \left(\alpha_l^{(i)}\right)^2+\left( \beta_l^{(i)}\right)^2 \rg=t_i^2$ and $t_i \geq 0$, for every $i=1 ,\dots ,r+s$;
\item $\prod_{i=1}^{r+s}t_i^{d_i} \leq t$.
\end{itemize}
It is clear that the set $S$ is defined by polynomial equalities and inequalities. In fact, the first condition is enforced by the fact that the coordinates of $\x_i$ are the images of $\bo{\alpha}^{(i)}$ and $ \bo{\beta}^{(i)}$ under the appropriate symmetric functions, which are polynomials. The second and the third conditions are also clearly obtained by polynomial equalities and inequalities. Therefore, $S$ is a semialgebraic set. The claim follows after noting that $Z$ is nothing but the projection of $S$ on the first $n(r+2s)+1$ coordinates and applying the Tarski-Seidenberg principle (Theorem \ref{tarski}).
\end{proof}

By Lemma 1.6.7 of \cite{BombGub}, there exists a positive real constant $\gamma\leq 1$ such that
$$
\gamma |\z|_\infty \leq M(\z), \text{ for every $\z \in \R^{\en+1}$ or $\C^{\en+1}$},
$$
where, if $\z=(z_0, \dots ,z_\en) \in \R^{\en+1}$ or $\C^{\en+1}$, $|\z|_\infty=\max \lg |z_0| ,\ldots ,|z_\en| \rg$ is the usual $\max$ norm. % Note that, since $|(1,0,\dots ,0)|_\infty = M(1,0,\dots ,0)=1$, $\gamma$ must be less than or equal to 1.
Clearly we have, for $\x \in \R^\en$
\begin{equation}\label{ineqnorm}
N(\x):=\gamma |(1,\x)|_\infty \leq \M(\x) 
\end{equation}
in the real case and, for the complex case,
\begin{equation}\label{ineqnorm2}
N(\x):= \gamma |(1,\x)|_\infty \leq \gamma |(1,\z)|_\infty \leq \M(\z)= \M(\x) ,
\end{equation}
where $\x=(x_1, \dots ,x_{2\en}) \in \R^{2\en}$ and $\z=(x_1 + i x_2 ,\dots , x_{2\en-1} + i x_{2\en})$.

Recall that, by the definition, the monic Mahler measure function assumes values greater than or equal to 1, therefore, if $(\bo{x}_1, \ldots , \bo{x}_{r+s}) \in Z_t$ then $\M(\x_i)^{d_i} \leq t$ for every $i$. Thus, $ |\x_i|^{d_i}_\infty \leq \frac{t}{\gamma^{d_i}}$ and this means that $Z_t$ is bounded for every $t \in \R$. 

Now we can apply Theorem \ref{counttheorem} to the family $Z$. If we set $Z(T)=Z_T$, we have
\begin{equation}\label{extim}
\left| |Z(T)\cap \Lambda|-\frac{\Vol(Z(T))}{\det \Lambda} \right|\leq \sum_{j=0}^{{mn}-1}C \frac{V_j(Z(T))}{\lambda_1\cdots \lambda_j},
\end{equation}
for every $T \in \R$, where $\Lambda$ is a lattice in $\R^{mn}$ and $C$ is a real constant independent of $\Lambda$ and $T$. Recall that $V_j(Z(T))$ is the sum of the $j$-dimensional volumes of the orthogonal projections of $Z(T)$ on every $j$-dimensional coordinate subspace of $\R^{mn}$ and $V_0(Z(T))=1$.

\section{Proof of Theorem \ref{mainthm}}

We fix a number field $\ks$ of degree $m$ over $\Q$. The ring of integers $\Oseen_\ks$ of $\ks$, embedded into $\R^{r+2s}$ via $\sigma=(\sigma_1 ,\dots ,\sigma_{r+s})$, is a lattice of full rank. We embed $(\Oseen_\ks)^\en$ in $\R^{mn}$ via $\bo{a}\mapsto (\sigma_1 (\bo{a}) ,\dots ,\sigma_{r+s}(\bo{a}))$, where the $\sigma_i$ are extended to $\ks^\en$. We want to count lattice points of $\Lambda=(\Oseen_\ks)^\en$ inside $Z(T)$.

\begin{lemma}\label{lemdet}
 
We have 
$$
\det \Lambda =\left( 2^{-s}  \sqrt{|\Delta_\ks|}\right)^{\en},
$$
and its first successive minimum is $\lambda_1 \geq 1$.

\end{lemma}

\begin{proof}
This is a special case of Lemma 5 of \cite{Masser2007}. 
\end{proof}

Now we need to calculate the volume of $Z(T)$. We do something more general. Suppose we have $r+s$ continuous functions $f_i : \R^{n_i} \rightarrow [1, \infty)$, $i=1, \dots ,r+s$ where $1\leq n_i \leq d_i \en$ for every $i$. We define
\begin{equation}\label{rball}
Z_i(T) =\{\x \in \R^{ n_i}: f_i(\x)\leq T\},
\end{equation}
for every $i=1,\dots ,r+s$. 
Suppose that, for every $i$, there exists a polynomial $p_i(X) \in \R[X]$ of degree $n_i$ such that the volume of $Z_i(T)$ is $p_i\left( T \right) $ for every $T\geq 1$. Let $C_i$ be the leading coefficient of $p_i$. Moreover, let 
$$
\widetilde{Z}(T)=\left\lbrace (\bo{x}_1, \ldots , \bo{x}_{r+s}) \in \R^{\sum n_i} : \prod_{i=1}^{r+s} f_i(\bo{x}_i)^{d_i}\leq T \right\rbrace.
$$
Note that, since $f_i(\x_i)\geq 1$ for every $i$, $\widetilde{Z}(T)$ is bounded for every $T$.

\begin{lemma}\label{lemvol}

Let $q=r+s-1$. Under the hypotheses and the notation from above, for every $T \geq 1$, we have
\begin{equation*}\label{eqvol}
\Vol \left( \widetilde{Z}(T)\right)=\widetilde{p} \left( T^\frac{1}{2},\log T \right),
\end{equation*}
where $\widetilde{p}(X,Y)\in \R[X,Y]$, $\deg_X \widetilde{p} \leq 2 \en$, $\deg_Y \widetilde{p} \leq q$. In the case $n_i= d_i \en$ for every $i=1,\dots ,r+s$, the coefficient of $X^{2\en} Y^q$ is $\frac{\en^q }{q!} \prod_{i=1}^{q+1} C_i$. If $n_i < d_i \en $ for some $i$ then the monomial $X^{2 \en} Y^q$ does not appear in $\widetilde{p}$.

\end{lemma}

\begin{proof}
We have
$$
V(T):=\Vol\left(\widetilde{Z}(T)\right)= \int_{\widetilde{Z}(T)} \td \x_1 \dots \td \x_{q+1}.
$$
We proceed by induction on $q$. If $q=0$ there is nothing to prove. Suppose $q>0$ and let
$$
\widetilde{Z}^{(q)}(T)=\left\lbrace (\bo{x}_1, \ldots , \bo{x}_{q}) \in \R^{n_1+\dots +n_{q}} : \prod_{i=1}^{q} f_i(\bo{x}_i)^{d_i}\leq T \right\rbrace.
$$
Then
$$
V(T)= \int_{Z_{q+1}\left(T^{\frac{1}{d_{q+1}}} \right)} \left( \int_{\widetilde{Z}^{(q)}\left( T f_{q+1}(\x_{q+1})^{-d_{q+1}} \right)} \td \x_1 \dots \td \x_q \right) \td \x_{q+1}.
$$
By the inductive hypothesis there exists $\widetilde{p}_q(X,Y)\in \R[X,Y]$ such that
\begin{gather*}
V(T)= \int_{Z_{q+1}\left(T^{\frac{1}{d_{q+1}}} \right)} \widetilde{p}_q\left(  \left(   \frac{T}{f_{q+1}(\x_{q+1})^{d_{q+1}}}  \right)^{\frac{1}{2}} , \log \left( \frac{T}{f_{q+1}(\x_{q+1})^{d_{q+1}}} \right) \right) \td \x_{q+1},
\end{gather*}
where $\widetilde{p}_q(X,Y)\in \R[X,Y]$, $\deg_X \widetilde{p}_q\leq 2 \en$, $\deg_Y \widetilde{p}_q\leq q-1$ and, if $n_i = d_i \en$ for every $i=1, \dots ,q$, the coefficient of $X^{2\en} Y^{q-1}$ is $\frac{\en^{q-1} }{(q-1)!} \prod_{i=1}^{q} C_i$. If not, that monomial does not appear.

By $\mc{L}^\en$, we indicate the Lebesgue measure on $\R^\en$. Since $f_{q+1}$ is a measurable function, we get 
$$
V(T)= \int_{\left[1,T^{\frac{1}{d_{q+1}}}\right] } \widetilde{p}_q\left(  \left(   \frac{T}{X^{d_{q+1}}}  \right)^{\frac{1}{2}} , \log \left( \frac{T}{X^{d_{q+1}}} \right) \right) \td \left( \mc{L}^{n_{q+1}}\circ  f_{q+1}^{-1} \right)(X),
$$
where we consider $\mc{L}^{n_{q+1}}\circ  f_{q+1}^{-1}$ as a measure on $\left[1,T^{\frac{1}{d_{q+1}}}\right]$. In particular for $(u,v] \subseteq \left[1,T^{\frac{1}{d_{q+1}}}\right]$, $$
\left(\mc{L}^{n_{q+1}}\circ  f_{q+1}^{-1}\right) ((u,v])=p_{q+1}(v) - p_{q+1}(u), 
$$ 
and $\left(\mc{L}^{n_{q+1}}\circ  f_{q+1}^{-1}\right) (\{1\})=p_{q+1}(1)$. Using 1.29 Theorem of \cite{Rudin}, we get
\begin{gather*}\label{V(T)}
V(T)= \int_{\left(1,T^{\frac{1}{d_{q+1}}}\right] } \widetilde{p}_q\left(  \left(   \frac{T}{X^{d_{q+1}}}  \right)^{\frac{1}{2}} , \log \left( \frac{T}{X^{d_{q+1}}} \right) \right) p'_{q+1}(X) \td  \mc{L}^{1} (X)+ \\
+\widetilde{p}_q\left(    T^{\frac{1}{2}} , \log  T\right) p_{q+1}(1), \nonumber
\end{gather*}
where $p_{q+1}'$ is the derivative of $p_{q+1}$.

For some integer $c\geq 0$ we put $L(X,c)=X^c$ in case $c>0$ and $L(X,0)=1$. 
Because of the linearity of the integral we are reduced to calculate
\begin{gather*}
\mc{I}(a,b,c) =\int_{1}^{T^{\frac{1}{d_{q+1}}}} X^a \left( \frac{T}{X^{d_{q+1}}}\right)^{\frac{b}{2}} L\left( \log \frac{T}{X^{d_{q+1}}},c \right) \td X= \\
=T^{\frac{b}{2}} \int_{1}^{T^{\frac{1}{d_{q+1}}}} X^{a-\frac{b}{2}d_{q+1}} L\left( \log T - \log \left( X^{d_{q+1}} \right) ,c\right)  \td X,
\end{gather*}
for some integers $a,b,c$, with $0\leq a \leq n_{q+1} -1$, $0 \leq b \leq 2 \en$ and $0\leq c \leq q-1$. 
We have three possibilities. If $a-\frac{b}{2}d_{q+1}=-1$, then 
\begin{gather*}
\mc{I}(a,b,c)=T^{\frac{b}{2}}\int_{1}^{T^{\frac{1}{d_{q+1}}}} X^{-1} L\left( \log T - \log \left( X^{d_{q+1}} \right) ,c\right)  \td X=\\
=\frac{1}{(c+1)d_{q+1}} T^{\frac{b}{2}} \left( \log T \right)^{c+1}.
\end{gather*}
If $a-\frac{b}{2}d_{q+1}\neq -1$ and $c=0$,
$$
\mc{I}(a,b,0)= 
\frac{T^{\frac{b}{2}} }{a-\frac{b}{2}d_{q+1}+1}\left( T^{\frac{a-\frac{b}{2}d_{q+1}+1}{d_{q+1}}}-1 \right)= \frac{ T^{\frac{a+1}{d_{q+1}}}-T^{\frac{b}{2}}}{a-\frac{b}{2}d_{q+1}+1}
$$
If $a-\frac{b}{2}d_{q+1}\neq -1$ and $c\neq 0 $, then 
$$
\mc{I}(a,b,c) =- \frac{T^{\frac{b}{2}}  \left( \log T \right)^c }{a-\frac{b}{2}d_{q+1}+1}+ \frac{c d_{q+1}}{a-\frac{b}{2}d_{q+1}+1}\mc{I}(a,b,c-1).
$$
Therefore, one can see that $\mc{I}(a,b,c)$ is a polynomial in $T^{\frac{1}{2}}$ and $\log T$. In particular $\mc{I}(a,b,c)=\widehat{p}(T^{\frac{1}{2}},\log T)$, where $\widehat{p}(X,Y) \in \R[X,Y]$, with $\deg_X\widehat{p}\leq 2n$ and $\deg_Y\widehat{p}\leq q$. Note that in the case $a= d_{q+1}n-1$, $b=2 \en$ and $c=q-1$, the coefficient of $X^{2 \en} Y^{q}$ is $\frac{1}{qd_{q+1}}$ and 0 for any other choice of $a,b$ and $c$. Therefore, the monomial $X^{2\en} Y^q$ does not appear in $\widehat{p}$ if either $n_{q+1}<d_{q+1}n$ or $X^{
2\en} Y^{q-1}$ does not appear in $\widetilde{p}_q$, i.e., if $n_i < d_i \en $ for some $i$. To conclude, recall that, in the case $n_i=d_i \en$ for every $i=1, \dots ,r+s$, $p'_{q+1}$ has leading coefficient $\en d_{q+1} C_{q+1}$ and the coefficient of $X^{
2\en} Y^{q-1}$ in $\widetilde{p}_q$ is $\frac{\en^{q-1} }{(q-1)!} \prod_{i=1}^{q} C_i$, thus, the coefficient in front of $\mc{I}(d_{q+1}n -1,2 \en,q-1)$ in $V(T)$ is $\frac{\en^{q} d_{q+1} }{(q-1)!} \prod_{i=1}^{q+1} C_i$.
%for some integers $a,b,c$, with $0\leq a \leq n_{q+1} -1$, $0 \leq b \leq 2 \en$ and $0\leq c \leq q-1$. Integrating by parts, one can see that $\mc{I}(a,b,c)$ is a polynomial in $T^{\frac{1}{2}}$ and $\log T$. In particular $\mc{I}(a,b,c)=\widehat{p}(T^{\frac{1}{2}},\log T)$, where $\widehat{p}(X,Y) \in \R[X,Y]$, with $\deg_X\widehat{p}\leq 2n$ and $\deg_Y\widehat{p}\leq q$. Note that in the case $a= d_{q+1}n-1$, $b=2 \en$ and $c=q-1$, the coefficient of $X^{2 \en} Y^{q}$ is $\frac{1}{qd_{q+1}}$ and 0 for any other choice of $a,b$ and $c$. Therefore, the monomial $X^{2\en} Y^q$ does not appear in $\widehat{p}$ if either $n_{q+1}<d_{q+1}n$ or $X^{
%2\en} Y^{q-1}$ does not appear in $\widetilde{p}_q$, i.e., if $n_i < d_i \en $ for some $i$. To conclude, recall that, in the case $n_i=d_i \en$ for every $i=1, \dots ,r+s$, $p'_{q+1}$ has leading coefficient $\en d_{q+1} C_{q+1}$ and the coefficient of $X^{
%2\en} Y^{q-1}$ in $\widetilde{p}_q$ is $\frac{\en^{q-1} }{(q-1)!} \prod_{i=1}^{q} C_i$.
\end{proof}

The volumes of the sets
\begin{equation}\label{ballr}
\{(z_1,\ldots ,z_\en ) \in \R^\en :M(1,z_1, \ldots ,z_\en )\leq T\}
\end{equation}
and
\begin{equation}\label{cball}
\{(z_1,\ldots ,z_\en ) \in \C^n:M(1,z_1, \ldots ,z_\en )^2\leq T\}
\end{equation}
were computed by Chern and Vaaler in \cite{Chern2001}. By (1.16) and (1.17) of \cite{Chern2001}, these volumes are, for every $T\geq 1$, polynomials $p_\R(T)$ and $p_\C(T)$ of degree $\en$ and leading coefficients, respectively,
$$
C_{\R,\en } = 2^{\en -M} \left( \prod_{l=1}^{M}\left( \frac{2l}{2l+1} \right)^{\en -2l} \right) \frac{\en^M}{M!},  \footnote{There is a misprint in  (1.16) of \cite{Chern2001}, $2^{-N}$ should read $2^{-M}$.}
$$
with $M=\lfloor \frac{\en -1}{2} \rfloor$, and 
$$
C_{\C,\en} = \pi^\en \frac{\en^\en }{\left( \en ! \right)^2}.
$$

Suppose $q=0$ and recall Lemma \ref{lemdet}. In this case $Z(T)$ corresponds to (\ref{ballr}) if $m=1$ or to (\ref{cball}) if $m=2$. We have
\begin{equation}\label{eqvolq0}
\frac{\Vol(Z(T))}{\det \Lambda}=\frac{ 2^{s\en }}{\left(  \sqrt{|\Delta_\ks|}\right)^{\en}} C_{\R,\en}^r C_{\C,\en}^s T^\en  + \frac{P(T)}{\left(  \sqrt{|\Delta_\ks|}\right)^{\en}},
\end{equation}
for every $T> 1$, where $P(X)\in \R[X]$ depends only on $n$, $r$ and $s$  and has degree at most $\en-1$.

\begin{corollary}\label{corvol}
Suppose $q>0$. We have, for $T> 1$,
\begin{equation}\label{eqvolq1}
\frac{\Vol(Z(T))}{\det \Lambda}=\frac{\en^q 2^{s\en }}{q!\left(  \sqrt{|\Delta_\ks|}\right)^{\en}} C_{\R,\en}^r C_{\C,\en}^s T^\en \left( \log T \right)^q + \frac{P\left(T^\frac{1}{2},\log T\right)}{\left(  \sqrt{|\Delta_\ks|}\right)^{\en}},
\end{equation}
where $P(X,Y)\in \R[X,Y]$ depends on $n$, $r$ and $s$, $\deg_X P \leq 2 \en$, $\deg_Y P \leq q$ and the coefficient of $X^{2\en} Y^q$ is 0.
\end{corollary}

\begin{proof}
By Lemma \ref{lemvol} and the result of Chern and Vaaler about the volumes of the sets defined in (\ref{ballr}) and (\ref{cball}), the volume of $Z(T)$ is $p(T^\frac{1}{2},\log T)$ where $p(X,Y)\in \R[X,Y]$, $\deg_X p \leq 2\en $, $\deg_Y p \leq q$ and the coefficient of $X^{2\en} Y^q$ is $\frac{\en^q }{q!} C_{\R,\en}^r C_{\C,\en}^s $.
\end{proof}

Therefore, recalling $|\Delta_\ks|$ and $\lambda_1, \dots , \lambda_{mn}$ are greater than or equal to 1, by (\ref{eqvolq0}) and Corollary \ref{corvol}, (\ref{extim}) becomes
\begin{equation}\label{extim2}
\left| |Z(T)\cap \Lambda|-\frac{\en^q 2^{s\en}}{q!\left(  \sqrt{|\Delta_\ks |}\right)^{\en}} C_{\R,\en}^r C_{\C,\en}^s T^\en \left( \log T \right)^q \right|\leq \sum_{j=0}^{{mn}-1}C V_j(Z(T))+Q(T),
\end{equation}
for every $T> 1$, where $Q(T)$ is the function of $T$ obtained from the polynomial $P$ of (\ref{eqvolq0}) or (\ref{eqvolq1}) substituting the coefficients with their absolute values. Note that $Q$ depends only on $\m$ and $\en$.

Now we want to find a bound for $V_j(Z(T))$. Recall that in (\ref{ineqnorm}) and (\ref{ineqnorm2}) we have defined a function $N(\x)= \gamma |(1,\x)|_\infty$ such that $N(\x)\leq \M(\x)$. Let 
$$
Z'(T)=\left\lbrace (\bo{x}_1, \ldots , \bo{x}_{r+s}) \in \R^{mn} : \prod_{i=1}^{r+s} N(\x_i)^{d_i}\leq T \right\rbrace .
$$
Each $(\bo{x}_1, \ldots , \bo{x}_{r+s})$ with $\prod_{i=1}^{r+s} \M(\bo{x}_i)^{d_i}\leq T $ satisfies $\prod_{i=1}^{r+s} N(\x_i)^{d_i}\leq T $. Therefore, we have $Z(T) \subseteq Z'(T)$ and $V_j(Z(T))\leq V_j(Z'(T))$.

Suppose $q=0$. This means that $\ks$ is either $\Q$ ($m=1$) or an imaginary quadratic field ($m=2$). In any case any projection of $Z'(T)$ to a $j$-dimensional coordinate subspace has volume $ \left(  \frac{2}{\gamma} \right)^j T^{\frac{j}{\m}} $ if $T\geq  \gamma^\m$, for every $j=1,\dots {mn}-1$. Therefore we obtain 
\begin{equation}\label{vol1}
V_j(Z(T)) \leq V_j\left(Z'(T)\right)\leq E T^{ \en -\frac{1}{\m}} ,
\end{equation}
for some real constant $E$ depending only on $\en$ and $\m$. This holds for every $T> 1$ since $\gamma \leq 1$.

Now suppose $q>0$.

\begin{lemma}\label{lemvolZ'}

For every $j=1, \dots , {mn}-1$, there exists $P_j(X,Y) \in \R [X,Y]$ whose coefficients depend only on $\m$ and $\en$, with $\deg_X P_j \leq 2 \en$, $\deg_Y P_j \leq q$, and the coefficient of $X^{2\en} Y^q$ is 0, such that, for every $T> 1$, we have
\begin{equation*}\label{vol2}
V_j(Z'(T))= P_j\left(T^{\frac{1}{2}}, \log T \right).
\end{equation*}

\end{lemma}

\begin{proof}
By definition, the projection of $Z'(T)$ on a $j$-dimensional coordinate subspace is just the intersection of $Z'(T)$ with such subspace. To each such subspace $\Sigma$ we can associate integers $n_1,\dots ,\n_{r+s}$ with $0\leq n_i \leq d_in$ such that $\Sigma$ is defined by setting $d_i n-n_i$ coordinates of each $\x_i$ to 0. Therefore we are in the situation of Lemma \ref{lemvol} because, after dividing by $\gamma$, we have, for every $i$ such that $n_i>0$, a continuous function  $f_i : \R^{n_i} \rightarrow [1, \infty)$, with $\sum n_i =j$. This gives rise to sets of the form (\ref{rball}), whose volumes are $2^{n_i}T^{n_i}$. Since $j<{mn}$, not all $n_i$ can be equal to $d_i n$. Therefore, by Lemma \ref{lemvol}, the volume of any such projection equals a polynomial with the desired property and we have the claim.
\end{proof}

Recall the definition of $\cM^\ks(e,\cH)$ that was given in Section \ref{sect2}. Clearly  $\left| \cM^\ks(e,\cH)\right|$ is the number of $\bo{a} \in \Oseen_\ks^\e$ with $\prod_{i=1}^{r+s}\M\left(\sigma_i(\bo{a})\right)^{d_i}\leq \cH^\m $, i.e., $|Z(\cH^\m)\cap \Oseen_\ks^\e|$.

By (\ref{extim2}), (\ref{vol1}) and Lemma \ref{lemvolZ'} we have, for every $\cH> 1$,
\begin{equation*}
\left| \left| \cM^\ks(e,\cH)\right|-\frac{\e^q \m^q 2^{s\e}}{q!\left(  \sqrt{|\Delta_\ks |}\right)^{\e}} C_{\R,\e}^r C_{\C,\e}^s \cH^{\m \e} \left( \log \cH \right)^q \right|\leq E(\cH),
\end{equation*}
with
\begin{equation*}
E(\cH)=\left\lbrace
\begin{array}{ll}
\sum_{i=0}^{2  \e }\sum_{j=0}^q E_{i,j} \cH^{\frac{\m i}{2}}(\log \cH )^j,  & \mbox{ if $q\geq 1$, }\\
\sum_{i=0}^{\m \e-1} E_i \cH^i, & \mbox{ if $q=0$, }
\end{array} \right.
\end{equation*}
where $E_{2  \e ,q}=0$ and all the coefficients depend on $\m$ and $\e$. 

Finally, it is clear that for every $\cH_0>1$ one can find a $D_0$ such that, for every $\cH \geq \cH_0$,
\begin{equation*}
E(\cH)\leq \left\lbrace
\begin{array}{ll}
D_0    \mc{H}^{\m\e} \left( \log  \mc{H}  \right)^{q-1} ,  & \mbox{ if $q\geq 1$, }\\
D_0 \mc{H}^{\m\e-1}  , & \mbox{ if $q=0$, }
\end{array} \right.
\end{equation*}
and we derive the claim of Theorem \ref{mainthm}.

\section*{Acknowledgments}

The author would like to thank Martin Widmer for sharing his ideas, for his constant encouragement and advice, Giulio Peruginelli, Robert Tichy and Jeffrey Vaaler for many useful discussions and the anonymous referee for providing valuable suggestions.

\bibliographystyle{amsplain}
\bibliography{bibliography.bib}

\end{document}